\documentclass{amsart}

\usepackage{amsfonts}
\usepackage{amsmath}
\usepackage{hyperref}

\newtheorem{theorem}{Theorem}
\theoremstyle{plain}

\numberwithin{equation}{section}

\begin{document}

\title[New Upper and Lower bounds for the Trapezoid inequality]
{New Upper and Lower bounds for the Trapezoid inequality of
absolutely continuous functions and Applications}

\author[M.W. Alomari]{Mohammad W. Alomari}

\address{Department of Mathematics,
Faculty of Science and Information Technology, Irbid National
University, 21110 Irbid, Jordan} \email{mwomath@gmail.com}

\date{\today}
\subjclass[2010]{26D15, 26D10}

\keywords{Hermite--Hadamard inequality, Trapezoid inequality}

\begin{abstract}
In this paper, new upper and lower bounds for the Trapezoid
inequality of absolutely continuous functions are obtained.
Applications to some special means are provided as well.
\end{abstract}

\maketitle

\section{Introduction}

Let $f:I\subseteq \mathbb{R}\rightarrow \mathbb{R}$ be a convex function
defined on the interval $I$ of real numbers and $a,\,b\in I$, with $a<b$.
The following inequality, known as \textit{Hermite--Hadamard inequality}
for convex functions, holds:
\begin{equation}
f\left( {\frac{{a+b}}{2}}\right) \leq \int_{a}^{b}{f\left( x\right) dx}\leq
\frac{{f\left( a\right) +f\left( b\right) }}{2}.  \label{1}
\end{equation}
On the other hand, a very related inequality to (\ref{1}) was known in literature as the `\emph{Trapezoid inequality}', which states that:  if $f:[a,b]\to \mathbb{R}$ is twice differentiable such that $
\left\| {f''} \right\|_\infty  : = \mathop {\sup }\limits_{t \in \left( {a,b} \right)} \left| {f''\left( t \right)} \right| < \infty
$, then
\begin{align}
\label{trapineq} \left| {\int_a^b {f\left( x \right)dx}  - \left(
{b - a} \right)\frac{{f\left( a \right) + f\left( b \right)}}{2}}
\right| \le \frac{{\left( {b - a} \right)^3 }}{{12}}\left\| {f''}
\right\|_\infty.
\end{align}
In recent years many authors have established several inequalities connected
to the Hermite-Hadamard's inequality. For recent results, refinements,
counterparts, generalizations and new Hermite-Hadamard-type inequalities under various assumptions for the functions involved the reader may refer to
\cite{M1} -- \cite{9b} and the references therein.

In this paper, new upper and lower bounds for the Trapezoid
inequality of absolutely continuous functions are established.
Applications to some special means are provided as well.

\section{The result}
\begin{theorem}
\label{thm1}Let $f:I\subset \mathbb{R}\rightarrow \mathbb{R}_+$ be an absolutely
continuous mapping on $I^{\circ },$ the interior of the interval
$I$, where $a,b\in I$ with $a<b$. Then there exists $x\in (a,b)$ such that the double inequality
\begin{multline}
\frac{{\left( {b - a} \right)^2 }}{{2M^2 }}\left[ {\frac{1}{{b - a}}\int_a^b {f\left( s \right)ds}  - \frac{{M^2 }}{{\left( {x - a} \right)\left( {b - x} \right)}}f\left( x \right)} \right]
\\
\le\frac{{f\left( b \right) + f\left( a \right)}}{2} - \frac{1}{{b - a}}\int_a^b {f\left( s \right)ds}
\\
\le \frac{{\left( {b - a} \right)^2 }}{{2m^2 }}\left[ {\frac{1}{{b - a}}\int_a^b {f\left( s \right)ds}  - \frac{{m^2 }}{{\left( {x - a} \right)\left( {b - x} \right)}}f\left( x \right)} \right]\label{eq2.1}
\end{multline}
holds, where
\begin{align*}
M: = \left[ {\frac{{b - a}}{2} + \left| {x - \frac{{a + b}}{2}} \right|} \right],
\end{align*}
and
\begin{align*}
m: =  \left[ {\frac{{b - a}}{2} - \left| {x - \frac{{a + b}}{2}} \right|} \right].
\end{align*}
\end{theorem}

\begin{proof}
Consider the function $F:[a,b] \to (0,\infty)$ defined by
\begin{align*}
 F\left( t \right) = \frac{1}{{b - t}}\int_t^b {f\left( s \right)ds}  - \frac{1}{{t - a}}\int_a^t {f\left( s \right)ds}
 \end{align*}
for all $t \in [a,b]$. Since $f$ is absolutely continuous on
$[a,b]$, then it is easy to see that $F$ is differentiable on
$(a,b)$. So that by the Mean-Value Theorem there exits $x \in
(a,b)$ such that
\begin{align}
F'\left( x \right) = \frac{{F\left( b \right) - F\left( a \right)}}{{b - a}}.\label{eq2.2}
\end{align}
Simple calculations yield that
\begin{align*}
F'\left( x \right) =\frac{1}{{\left( {b - x} \right)^2 }}\int_x^b {f\left( s \right)ds}  + \frac{1}{{\left( {x - a} \right)^2 }}\int_a^x {f\left( s \right)ds}  - \frac{{\left( {b - a} \right)f\left( x \right)}}{{\left( {x - a} \right)\left( {b - x} \right)}}
\end{align*}
and
\begin{align*}
\frac{{F\left( b \right) - F\left( a \right)}}{{b - a}} &= \frac{1}{{b - a}}\left\{ {f\left( b \right) - \frac{1}{{b - a}}\int_a^b {f\left( s \right)ds}  - \left[ {\frac{1}{{b - a}}\int_a^b {f\left( s \right)ds}  - f\left( a \right)} \right]} \right\}
\\
&=\frac{1}{{b - a}}\left\{ {f\left( a \right) + f\left( b \right) - \frac{2}{{b - a}}\int_a^b {f\left( s \right)ds} } \right\}
\end{align*}
therefore, we have
\begin{align}
 F'\left( x \right) &=\frac{1}{{\left( {b - x} \right)^2 }}\int_x^b {f\left( s \right)ds}  + \frac{1}{{\left( {x - a} \right)^2 }}\int_a^x {f\left( s \right)ds}  - \frac{{\left( {b - a} \right)f\left( x \right)}}{{\left( {x - a} \right)\left( {b - x} \right)}}
\nonumber\\
&= \frac{2}{{b - a}}\left[ {\frac{{f\left( b \right) + f\left( a \right)}}{2} - \frac{1}{{b - a}}\int_a^b {f\left( s \right)ds} } \right].\label{eq2.3}
\end{align}
Now, for $x\in(a,b)$, we set
\begin{align*}
M: = \max \left\{ {x - a,b - x} \right\} = \left[ {\frac{{b - a}}{2} + \left| {x - \frac{{a + b}}{2}} \right|} \right],
\end{align*}
and
\begin{align*}
m: = \min \left\{ {x - a,b - x} \right\} = \left[ {\frac{{b - a}}{2} - \left| {x - \frac{{a + b}}{2}} \right|} \right],
\end{align*}
so that, since $f$ is positive we have
\begin{align*}
\frac{1}{{M^2 }}\int_a^x {f\left( s \right)ds}  \le \frac{1}{{\left( {x - a} \right)^2 }}\int_a^x {f\left( s \right)ds}  \le \frac{1}{{m^2 }}\int_a^x {f\left( s \right)ds}
\end{align*}
and
\begin{align*}
\frac{1}{{M^2 }}\int_x^b {f\left( s \right)ds}  \le \frac{1}{{\left( {b - x} \right)^2 }}\int_x^b {f\left( s \right)ds}  \le \frac{1}{{m^2 }}\int_x^b {f\left( s \right)ds}
\end{align*}
By adding the above two inequalities we get
\begin{align}
\frac{1}{{M^2 }}\int_a^b {f\left( s \right)ds}  \le \frac{1}{{\left( {b - x} \right)^2 }}\int_x^b {f\left( s \right)ds}  + \frac{1}{{\left( {x - a} \right)^2 }}\int_a^x {f\left( s \right)ds}  \le \frac{1}{{m^2 }}\int_a^b {f\left( s \right)ds} \label{eq2.4}
\end{align}
and by (\ref{eq2.3}) and (\ref{eq2.4}) we get
\begin{align*}
&\frac{1}{{M^2 }}\int_a^b {f\left( s \right)ds}  - \frac{{\left(
{b - a} \right)f\left( x \right)}}{{\left( {x - a} \right)\left(
{b - x} \right)}}
\\
&\le \frac{2}{{b - a}}\left[ {\frac{{f\left( b \right) + f\left( a
\right)}}{2} - \frac{1}{{b - a}}\int_a^b {f\left( s \right)ds} }
\right]
\\
&=\frac{1}{{\left( {b - x} \right)^2 }}\int_c^b {f\left( s \right)ds}  + \frac{1}{{\left( {x - a} \right)^2 }}\int_a^x {f\left( s \right)ds}  - \frac{{\left( {b - a} \right)f\left( x \right)}}{{\left( {x - a} \right)\left( {b - x} \right)}}
\\
&\le \frac{1}{{m^2 }}\int_a^b {f\left( s \right)ds} - \frac{{\left( {b - a} \right)f\left( x \right)}}{{\left( {x - a} \right)\left( {b - x} \right)}}.
\end{align*}
Hence, by multiplying the above inequality by the quantity $\frac{{b - a}}{{2}}$ we get
\begin{align*}
&\frac{{b - a}}{2}\left[ {\frac{1}{{M^2 }}\int_a^b {f\left( s \right)ds}  - \frac{{\left( {b - a} \right)f\left( x \right)}}{{\left( {x - a} \right)\left( {b - x} \right)}}} \right]
\\
&\le\frac{{f\left( b \right) + f\left( a \right)}}{2} - \frac{1}{{b - a}}\int_a^b {f\left( s \right)ds}
\\
&\le \frac{{b - a}}{2}\left[ {\frac{1}{{m^2 }}\int_a^b {f\left( s
\right)ds}  - \frac{{\left( {b - a} \right)f\left( x
\right)}}{{\left( {x - a} \right)\left( {b - x} \right)}}}
\right].
\end{align*}
 Rearranging the terms we may write,
\begin{align*}
&\frac{{\left( {b - a} \right)^2 }}{{2M^2 }}\left[ {\frac{1}{{b - a}}\int_a^b {f\left( s \right)ds}  - \frac{{M^2 }}{{\left( {x - a} \right)\left( {b - x} \right)}}f\left( x \right)} \right]
\\
&\le\frac{{f\left( b \right) + f\left( a \right)}}{2} - \frac{1}{{b - a}}\int_a^b {f\left( s \right)ds}
\\
&\le \frac{{\left( {b - a} \right)^2 }}{{2m^2 }}\left[ {\frac{1}{{b - a}}\int_a^b {f\left( s \right)ds}  - \frac{{m^2 }}{{\left( {x - a} \right)\left( {b - x} \right)}}f\left( x \right)} \right],
\end{align*}
for some $x \in (a,b)$; which proves the inequality (\ref{eq2.1}).
\end{proof}

Here, its convenient to note that
\begin{align}
0\le\Delta &:=\text{The right-hand side of (\ref{eq2.1})} - \text{The left-hand side of (\ref{eq2.1})}
\nonumber\\
&=\frac{{\left( {b - a} \right)^2 }}{{2m^2 }}\left[ {\frac{1}{{b - a}}\int_a^b {f\left( s \right)ds}  - \frac{{m^2 }}{{\left( {x - a} \right)\left( {b - x} \right)}}f\left( x \right)} \right]
\nonumber\\
&\qquad- \frac{{\left( {b - a} \right)^2 }}{{2M^2 }}\left[ {\frac{1}{{b - a}}\int_a^b {f\left( s \right)ds}  - \frac{{M^2 }}{{\left( {x - a} \right)\left( {b - x} \right)}}f\left( x \right)} \right]\nonumber
\\
&= \left( {\frac{{M^2  - m^2 }}{{2m^2 M^2 }}} \right)\left( {b - a} \right)\int_a^b {f\left( s \right)ds},
\end{align}
thus, it is clear that $\left( {\frac{{M^2  - m^2 }}{{2m^2 M^2 }}} \right) \ge0 $ and so that the difference $\Delta\ge0$ iff $f(t)\ge0$ $\forall t \in [a,b]$.

Finally, we note that another interesting form of the inequality (\ref{eq2.1}) may be deduced by rewriting the terms of (\ref{eq2.1}), to get:
\begin{align}
&\frac{{2M^2 \left( {b - a} \right)}}{{\left( {2M^2  + \left( {b - a} \right)^2 } \right)}}\left[ {\frac{{\left( {b - a} \right)^2 }}{{2\left( {x - a} \right)\left( {b - x} \right)}}f\left( x \right) + \frac{{f\left( a \right) + f\left( b \right)}}{2}} \right]
\nonumber\\
&\ge \int_a^b {f\left( s \right)ds}
\label{eq2.6}\\
&\ge \frac{{2m^2 \left( {b - a} \right)}}{{\left( {2m^2  + \left( {b - a} \right)^2 } \right)}}\left[ {\frac{{\left( {b - a} \right)^2 }}{{2\left( {x - a} \right)\left( {b - x} \right)}}f\left( x \right) + \frac{{f\left( a \right) + f\left( b \right)}}{2}} \right],\nonumber
\end{align}
and so that, we have
\begin{align}
0 &\le \int_a^b {f\left( s \right)ds} - \frac{{2m^2 \left( {b - a} \right)}}{{\left( {2m^2  + \left( {b - a} \right)^2 } \right)}}\Psi _f \left( {a,b;x} \right)
\label{eq2.7}\\
&\le \left[ {\frac{{2M^2 \left( {b - a} \right)}}{{\left( {2M^2  + \left( {b - a} \right)^2 } \right)}}-\frac{{2m^2 \left( {b - a} \right)}}{{\left( {2m^2  + \left( {b - a} \right)^2 } \right)}}} \right]\Psi _f \left( {a,b;x} \right)\nonumber
\end{align}
where,
\begin{align*}
\Psi _f \left( {a,b;x} \right): = \frac{{\left( {b - a} \right)^2 }}{{2\left( {x - a} \right)\left( {b - x} \right)}}f\left( x \right) + \frac{{f\left( a \right) + f\left( b \right)}}{2}
\end{align*}
for some $x \in (a,b)$.

In an interesting particular case, let $\mathcal{F}$ be the set of
all functions $f:I\subset \mathbb{R}\rightarrow \mathbb{R}_+$ that
satisfy the assumptions of Theorem \ref{thm1} such that the
required $x \in (a,b)$ is $x= \frac{a+b}{2}$, (in this case we
have $M=m=\frac{b-a}{2}$) thus from (\ref{eq2.7}) every such $f$
satisfies that
\begin{align}
\int_a^b {f\left( s \right)ds}  = \frac{1}{3}\left( {b - a}
\right)\left[ {2f\left( {\frac{{a + b}}{2}} \right) +
\frac{{f\left( a \right) + f\left( b \right)}}{2}}
\right],\label{eq2.8}
\end{align}
where,
\begin{align*}
\Psi _f \left( {a,b;{\textstyle{{a + b} \over 2}}}
\right)=2f\left( {\frac{{a + b}}{2}} \right) + \frac{{f\left( a
\right) + f\left( b \right)}}{2}.
\end{align*}
 On the other hand, it is well-known that the error term in Simpson's
 quadrature rule (\ref{Simpson1}) involves a fourth derivatives, however using
 the above
 observation and for every $f\in \mathcal{F}$;
$\int_a^b {f\left( s \right)ds}$ can be evaluated `exactly' using
the Simpson formula (\ref{eq2.8}) with no errors, and without
going through its higher derivatives which may not exists or hard
to find; as in the classical result.
\begin{align}
\label{Simpson1}\int\limits_a^b {f\left( x \right)dx}
=\frac{b-a}{3}\left[ {\frac{{f\left( a \right) + f\left( b
\right)}}{2} + 2 f\left( {\frac{{a + b}}{2}} \right)} \right]
 + \frac{\left( {b - a} \right)^5}{{90}}\left\|
{f^{\left( 4 \right)} } \right\|_\infty.
\end{align}

\section{Applications to means}

A function $M: \mathbb{R}^2_+\to \mathbb{R}_+ $,
is called a Mean function if it has the following properties:
\begin{enumerate}
\item Homogeneity: $M\left( {ax,ay} \right) = aM\left( {x,y}
\right)$, for all $a > 0$,

\item Symmetry : $M\left( {x,y} \right) = M\left( {y,x} \right)$,

\item Reflexivity : $M\left( {x,x} \right) = x$,

\item Monotonicity: If $x\le x'$ and $y \le y'$, then $M\left(
{x,y} \right) \le M\left( {x',y'} \right)$,

\item Internality: $\min\{{x,y}\} \le M\left( {x,y} \right) \le
\max\{{x,y}\}$.
\end{enumerate}
We shall consider the means for arbitrary positive real numbers
$\alpha, \beta$ $(\alpha \ne \beta)$, see \cite{Bullen1}--\cite{Bullen2}.
We take
\begin{enumerate}
\item The arithmetic mean :
$$A := A\left( {\alpha ,\beta } \right) = \frac{{\alpha  + \beta
}}{2},\,\,\,\,\,\alpha ,\beta \in  \mathbb{R}_+.$$

\item The geometric mean :
$$G: = G\left( {\alpha ,\beta } \right) = \sqrt {\alpha \beta },\,\,\,\,\,\alpha ,\beta \in \mathbb{R}_+$$

\item The harmonic mean :
$$H: = H\left(
{\alpha ,\beta } \right) = \frac{2}{{\frac{1}{\alpha } +
\frac{1}{\beta }}},\,\,\,\,\,\alpha ,\beta \in \mathbb{R}_+
 - \left\{ 0 \right\}.$$

\item The power mean : $$M_r \left( {\alpha ,\beta } \right) =
\left( {\frac{{\alpha ^r  + \beta ^r }}{2}} \right)^{{\textstyle{1
\over r}}},\,\,\,\,\, r \ge 1,\,\, \alpha ,\beta \in
\mathbb{R}_+$$

\item The identric mean:
$$I\left( {\alpha ,\beta } \right) = \left\{ \begin{array}{l}
 \frac{1}{e}\left( {\frac{{\beta ^\beta  }}{{\alpha ^\alpha  }}} \right)^{{\textstyle{1 \over {\beta  - \alpha }}}} ,\,\,\,\,\,\,\,\alpha  \ne \beta  \\
 \alpha ,\,\,\,\,\,\,\,\,\,\,\,\,\,\,\,\,\,\,\,\,\,\,\,\,\,\,\,\,\,\,\,\,\,\alpha  = \beta  \\
 \end{array} \right.,\,\,\,\,\,\alpha ,\beta>0$$

\item The logarithmic mean :
$$L := L\left( {\alpha ,\beta } \right) = \frac{{\alpha  - \beta }}{{\ln
\left| \alpha  \right| - \ln \left| \beta  \right|}},\,\,\left|
\alpha  \right| \ne \left| \beta  \right|,\,\,\alpha, \beta  \ne
0,\,\,\alpha ,\beta \in \mathbb{R}_+ .$$

\item The generalized log-mean:
$$L_p:=L_p \left( {\alpha ,\beta } \right) = \left[
{\frac{{\beta ^{p + 1}  - \alpha ^{p + 1} }}{{\left( {p + 1}
\right)\left( {\beta  - \alpha } \right)}}} \right]^{{\textstyle{1
\over p}}} ,\,p \in \mathbb{R}\backslash \left\{ { -1 ,0}
\right\},\,\alpha ,\beta >0.
$$
\end{enumerate}
It is well known that $L_p$ is monotonic nondecreasing over $p \in
\mathbb{R}$, with $L_{-1} := L$ and $L_0 := I$. In particular, we
have the following inequality $H \le G \le L \le I \le A$.\\

As a direct example on the inequality (\ref{eq2.1}), consider $f:[a,b] \subset(0,\infty) \to (0,\infty)$ given by $f\left( s \right) = \frac{1}{{s^2 }}$, , $s \in [a,b]$, it is easy to see that
$F\left( t \right) = \frac{{b - a}}{{ab}}\cdot\frac{1}{{t^2 }}$, and so that the required $x \in (a,b)$ is $x =
\sqrt {ab}
:=G\left( {a,b} \right)$.
Applying (\ref{eq2.1}), we get
\begin{multline}
\frac{{\left( {b - a} \right)^2 }}{{2M^2 }}\left[ {1 - \frac{{M^2 }}{{\left( {G\left( {a,b} \right) - a} \right)\left( {b - G\left( {a,b} \right)} \right)}}} \right]
\le \frac{G^2 \left( {a,b} \right)}{{H\left( {a^2,b^2} \right)}} - 1
\\
\le \frac{{\left( {b - a} \right)^2 }}{{2m^2 }}\left[ {1 - \frac{{m^2 }}{{\left( {G\left( {a,b} \right) - a} \right)\left( {b - G\left( {a,b} \right)} \right)}}} \right],
\end{multline}
where,
$M: = \left[ {\frac{{b - a}}{2} + \left| {G\left( {a,b} \right) - A\left( {a,b} \right)} \right|} \right]$, and $m: = \left[ {\frac{{b - a}}{2} - \left| {G\left( {a,b} \right) - A\left( {a,b} \right)} \right|} \right]$.\\

In general, the reader may check that it is not easy to find the value of $x$ that satisfies the inequality (\ref{eq2.1}). For example, we consider $f:[a,b] \subset(0,\infty) \to (0,\infty)$ given by

(1-) $f\left( s \right) = \frac{1}{s}$,  $s \in [a,b] \subset (0,\infty)$, so that
$F\left( t \right) = \ln \left( {\frac{{\left( {b/t} \right)^{1/\left( {b - t} \right)} }}{{\left( {t/a} \right)^{1/\left( {t - a} \right)} }}} \right)$ .
Applying (\ref{eq2.1}), we get
\begin{multline}
\frac{{\left( {b - a} \right)^2 }}{{2M^2 }}\left[ {\frac{1}{{L\left( {a,b} \right)}} - \frac{{M^2 }}{{\left( {x - a} \right)\left( {b - x} \right)}}\frac{1}{x}} \right]
\le \frac{1}{{H\left( {a,b} \right)}}
 - \frac{1}{{L\left( {a,b} \right)}}
\\
\le \frac{{\left( {b - a} \right)^2 }}{{2m^2 }}\left[ {\frac{1}{{L\left( {a,b} \right)}} - \frac{{m^2 }}{{\left( {x - a} \right)\left( {b - x} \right)}}\frac{1}{x}} \right]
\end{multline}
where $x\in (a,b)$ satisfying the equation (\ref{eq2.2}), and $m,M$ are defined above.\\

(2-) $f\left( s \right) = \ln(s)$,  $s \in [a,b] \subset (0,\infty)$, so that
$$F\left( t \right) =
\ln \left( {\frac{{\left( {b^b /t^t } \right)^{1/\left( {b - t} \right)} }}{{\left( {t^t /a^a } \right)^{1/\left( {t - a} \right)} }}} \right)= \ln \left( {\frac{{I\left( {t,b} \right)}}{{I\left( {a,t} \right)}}} \right)
,\,\,\,\,t \in [a,b]$$  where $I(\cdot,\cdot)$ is the identric mean. Now,
applying (\ref{eq2.1}), we get
\begin{multline}
\frac{{\left( {b - a} \right)^2 }}{{2M^2 }}\left[ {\ln I\left( {a,b} \right) - \frac{{M^2 \ln x}}{{\left( {x - a} \right)\left( {b - x} \right)}}} \right]
\le \ln G\left( {a,b} \right) - \ln I\left( {a,b} \right)
\\
\le \frac{{\left( {b - a} \right)^2 }}{{2m^2 }}\left[ {\ln I\left( {a,b} \right) - \frac{{m^2 \ln x}}{{\left( {x - a} \right)\left( {b - x} \right)}}} \right]
\end{multline}
where $x\in (a,b)$ satisfying the equation (\ref{eq2.2}), and $m,M$ are defined above.\\

(3-) $f\left( s \right) = s^p$, $s \in [a,b] \subset (0,\infty)$ and $p \in \mathbb{R}\backslash \left\{ { -1 ,0}
\right\}$, so that
$$F\left( t \right) =
\frac{{b^{p + 1}  - t^{p + 1} }}{{\left( {b - t} \right)\left( {p + 1} \right)}} - \frac{{t^{p + 1}  - a^{p + 1} }}{{\left( {t - a} \right)\left( {p + 1} \right)}} = L_p^p \left( {t,b} \right) - L_p^p \left( {a,t} \right)
,\,\,\,\,t \in [a,b]$$where $L_p^p(\cdot,\cdot)$ is the generalized logarithmic mean. Applying (\ref{eq2.1}), we get
\begin{multline}
\frac{{\left( {b - a} \right)^2 }}{{2M^2 }}\left[ {L_p^p \left( {a,b} \right) - \frac{{M^2 x^p }}{{\left( {x - a} \right)\left( {b - x} \right)}}} \right]
\le  M_p^p\left( {a,b} \right) -
L_p^p \left( {a,b} \right)
\\
\le \frac{{\left( {b - a} \right)^2 }}{{2m^2 }}\left[ {L_p^p \left( {a,b} \right) - \frac{{m^2 x^p }}{{\left( {x - a} \right)\left( {b - x} \right)}}} \right]
\end{multline}
where $x\in (a,b)$ satisfying the equation (\ref{eq2.2}), and $m,M$ are defined above.

\centerline{}

\centerline{}


\begin{thebibliography}{9}
\setlength{\itemsep}{1pt}

\bibitem{M1}
M.W. Alomari, A companion of the generalized
trapezoid inequality and applications, \textit{Journal of
Math. Appl.},  36 (2013), 5--15.

\bibitem{M2}
M.W. Alomari, New sharp inequalities of Ostrowski
and generalized trapezoid type for the Riemann--Stieltjes
integrals and applications, \textit{Ukrainian Mathematical
Journal},  65 (7) (2013), 995--1018.

\bibitem{M3}
M.W. Alomari, M. Darus and U.S. Kirmaci, Some
inequalities of Hermite-Hadamard type for $s$-convex functions,
\textit{Acta Mathematica Scientia}, 31 B(4) (2011) : 1643--1652.

\bibitem{M4}
M.W. Alomari, M. Darus and U. Kirmaci, Refinements of
Hadamard-type inequalities for quasi-convex functions with
applications to trapezoidal formula and to special means,
\textit{Comp. Math. Appl.}, 59 (2010), 225--232.

\bibitem{M5}
M. Alomari and M. Darus, On the Hadamard's
inequality for log-convex functions on the coordinates, \textit{J.
Ineq. Appl.}, 2009, Article ID 283147, 13 pages,
doi:10.1155/2009/283147.

\bibitem{Bullen1}
P. S. Bullen, D. S. Mitrinovi\'{c} and M. Vasi\'{c}", Means and Their Inequalities,
Dordrecht: Kluwer Academic, 1988.

\bibitem{Bullen2}
P. S. Bullen, Handbook of Means and Their Inequalities,
Dordrecht: Kluwer Academic, 2003.

\bibitem{1b} S.S. Dragomir, Two mappings in connection to Hadamard's
inequalities, \emph{J. Math. Anal. Appl.}, {167} (1992) 49--56.

\bibitem{2b} S.S. Dragomir and R.P. Agarwal, Two inequalities for
differentiable mappings and applications to special means of real numbers
and to trapezoidal formula, \emph{Appl. Math. Lett.}, {11} (1998) 91--95.

\bibitem{b} S. S. Dragomir and C. E. M. Pearce, ``Selected Topics on
Hermite-Hadamard Inequalities and Applications,'' RGMIA
Monographs, Victoria University, 2000,

http://www.staff.vu.edu.au/RGMIA/ monographs/hermite
hadamard.html.

\bibitem{3b} S.S. Dragomir, Y.J. Cho and S.S. Kim, Inequalities of
Hadamard's type for Lipschitzian mappings and their applicaitions, \emph{J.
Math. Anal. Appl.}, 245 (2000), 489--501.

\bibitem{4b} D.A. Ion, Some estimates on the Hermite-Hadamard inequality
through quasi-convex functions, Annals of University of Craiova, \emph{Math.
Comp. Sci. Ser., }{34} (2007), 82--87.

\bibitem{5b} {U.S. Kirmaci, Inequalities for differentiable mappings and
applicatios to special means of real numbers to midpoint formula, \emph{%
Appl. Math. Comp.},} {147} (2004), 137--146.

\bibitem{6b} {U.S. Kirmaci and M.E. \"{O}zdemir, On some inequalities for
differentiable mappings and applications to special means of real numbers
and to midpoint formula, \emph{Appl. Math. Comp.},} {153} (2004), 361--368.

\bibitem{7b} {M.E. \"{O}zdemir, A theorem on mappings with bounded
derivatives with applications to quadrature rules and means, \emph{Appl.
Math. Comp.},} {138} (2003), 425--434.

\bibitem{8b} C.E.M. Pearce and J. Pe\v{c}ari\'{c}, Inequalities for
differentiable mappings with application to special means and quadrature
formula, \emph{Appl. Math. Lett.,} 13 (2000) 51--55.

\bibitem{9b} G.S. Yang, D.Y. Hwang and K.L. Tseng, Some inequalities for
differentiable convex and concave mappings, \emph{Comp. Math. Appl., }{47}
(2004), 207--216.

\end{thebibliography}
\end{document}